\documentclass[11pt]{amsart}

\usepackage{amsmath,amsfonts,amssymb,amsthm}
\usepackage{latexsym}
\usepackage{eufrak}
\usepackage{euscript}
\usepackage{dsfont}
\usepackage{enumerate}
\usepackage{array}
\usepackage{booktabs}
\usepackage{tabularx}
\usepackage{tikz}
\usepackage{caption}
\usepackage{color}
\usepackage{wasysym}
\usepackage{hyperref}

\newtheorem{theorem}{Theorem}[section]
\newtheorem{lemma}[theorem]{Lemma}

\newtheorem{definition}[theorem]{Definition}

\numberwithin{equation}{section}
\usepackage{natbib}

\begin{document}
\bibliographystyle{abbrv}

\title[Stable Lane--Emden Equations on Graphs]
{Nonexistence of Stable Solutions to Lane--Emden Equations on Graphs with Polynomial Volume Growth}

\author{Weiqi Guan}
\address{Weiqi Guan: School of Mathematical Sciences, Fudan University, Shanghai 200433, China}
\email{wqguan24@m.fudan.edu.cn}

\begin{abstract}
We study stable solutions of the Lane--Emden equation
\[
-\Delta u=|u|^{p-1}u
\]
on graphs with polynomial volume growth. Due to the lack of a chain rule in the discrete setting, we establish suitable algebraic inequalities which serve as substitutes for the usual differential identities. Combining these inequalities with the stability condition and cutoff arguments, we obtain nonexistence results for stable solutions in a range of exponents analogous to the Joseph--Lundgren range. As an application, we derive a lower bound for the Sobolev best constant in the discrete setting.
\end{abstract}

\maketitle

\section{Introduction}

The study of nonlinear elliptic equations has played a central role in the development of modern analysis and the theory of partial differential equations. A prototypical example in semilinear elliptic theory is the Lane--Emden equation
\begin{equation}\label{eq:lame}
-\Delta u = |u|^{p-1}u.
\end{equation}

This equation also has a natural geometric origin arising from conformal geometry. Indeed, under a conformal deformation of the form
\[
\tilde{g} = u^{\frac{4}{n-2}} g,
\]
the scalar curvature equation reduces to a semilinear elliptic equation with critical growth; see
\cite{yamabe1960deformation,trudinger1968remarks,aubin1976equations,schoen1984conformal,lee1987yamabe}. 
In the Euclidean setting, this leads to the critical Lane--Emden equation
\[
-\Delta u = u^{\frac{n+2}{n-2}},
\]
which corresponds to the Sobolev critical exponent.

In $\mathbb{R}^d$, the existence and qualitative properties of solutions to \eqref{eq:lame} are closely tied to certain critical exponents, such as the Sobolev exponent and the Joseph--Lundgren exponent \cite{joseph1973quasilinear,gidas1981global}. Of particular interest are \emph{stable solutions}, namely those for which the second variation of the associated energy functional is nonnegative. The classification and nonexistence of stable solutions have been extensively studied; see, for example, \cite{farina2007classification,farina2007stable,dupaigne2010stable}.

A key role in this theory is played by the Joseph--Lundgren exponent
\begin{equation}
p_{\mathrm{JL}} =
\begin{cases}
\displaystyle \frac{d^2 - 8d + 4 + 8\sqrt{d-1}}{(d-2)(d-10)}, & d \ge 11, \\[6pt]
+\infty, & 3 \le d \le 10,
\end{cases}
\end{equation}
which sharply characterizes the existence of stable solutions. More precisely, it was shown in \cite{farina2007classification} that equation \eqref{eq:lame} admits stable solutions with $1<p<\infty$ if and only if $p \geq p_{\mathrm{JL}}$.

In recent years, there has been increasing interest in extending such analytical results to discrete settings, particularly to graphs; see, e.g., \cite{hua2021existence}. However, many fundamental tools from the continuous theory, such as the chain rule and Pohozaev-type identities, are no longer available on graphs. This absence creates substantial challenges in transferring classification and nonexistence results from the Euclidean setting to discrete structures. Consequently, nonlinear elliptic equations and variational problems on graphs have attracted considerable attention; see, for instance, \cite{gu2023semi}.

The purpose of the present paper is to investigate stable solutions of \eqref{eq:lame} on graphs with polynomial volume growth. Our main result can be stated as follows.

\begin{theorem}[Nonexistence of stable solutions]\label{Thm:lattice}
Let $G$ be a graph with polynomial volume growth of order $d$, and let $u \in C(G)$ be a stable solution to
\[
-\Delta u = |u|^{p-1}u
\]
on $G$. Then $u \equiv 0$ provided that
\[
\begin{cases}
1 < p < \infty, & \text{if } d \le 10, \\[4pt]
1 < p < p_c(d) := \displaystyle
\frac{(d-2)^2 - 4d + 8\sqrt{d-1}}{(d-2)(d-10)},
& \text{if } d \ge 11.
\end{cases}
\]
\end{theorem}

Theorem \ref{Thm:lattice} can be viewed as a discrete analogue of the nonexistence result established in \cite{farina2007classification}. In contrast to the Euclidean setting, however, the existence of stable solutions for $p \geq p_c(d)$ on graphs remains an open problem.

The proof of Theorem \ref{Thm:lattice} relies on two main ingredients. First, we establish a family of algebraic inequalities that serve as substitutes for the chain rule, which is not available in the discrete setting. Second, we combine these inequalities with the stability condition and a cutoff argument tailored to graphs with polynomial volume growth.

\section{Preliminaries}

Let $G=(V,E)$ be an infinite, locally finite, connected and simple graph. We denote by $C(G)$ the space of real-valued functions on $V$, and by $C_c(G)$ the subspace of functions with finite support.

The graph Laplacian is defined by
\[
\Delta f(x)=\sum_{y\sim x}(f(y)-f(x)).
\]

For $1\le p<\infty$, the $l^p$ norm is defined by
\[
\|u\|_{l^p(G)}
=
\left(\sum_{x\in V}|u(x)|^p\right)^{1/p},
\]
and
\[
\|u\|_{l^\infty(G)}
=
\sup_{x\in V}|u(x)|.
\]
When $u\in l^1(G)$, we write
\[
\int_G u\,d\mu=\sum_{x\in V}u(x).
\]

For $u\in C(G)$, define
\[
|\nabla_G u|_p(x)
=
\left(
\frac12\sum_{y\sim x}|u(y)-u(x)|^p
\right)^{1/p}.
\]
In particular,
\[
\int_G |\nabla_G u|^2\,d\mu
=
\frac12\sum_{x\in V}\sum_{y\sim x}|u(y)-u(x)|^2.
\]

We also use the notation
\[
\nabla_{xy}u=u(y)-u(x).
\]

For functions $u,v\in C_c(G)$, the discrete integration by parts formula reads
\[
\int_G (-\Delta u)v\,d\mu
=
\frac12\sum_{x\in V}\sum_{y\sim x}
(u(x)-u(y))(v(x)-v(y)).
\]

\begin{definition}
Let $p>1$. A solution $u\in C(G)$ of
\[
-\Delta u=|u|^{p-1}u
\]
is called stable if
\[
p\int_G |u|^{p-1}\phi^2\,d\mu
\le
\int_G |\nabla \phi|^2\,d\mu
\]
for every $\phi\in C_c(G)$.
\end{definition}

\begin{definition}
We say that $G$ has polynomial volume growth of order $d$ if there exists a constant $C>0$ such that
\[
|B_R(x)|\le CR^d
\]
for all $x\in V$ and all $R\ge 1$.
\end{definition}

\section{Algebraic Inequalities}

Due to the loss of the chain rule in the discrete setting, we need the following elementary inequalities.

\begin{lemma}\label{Lma:ineq}
For $a,b\in\mathbb{R}$ and $\gamma\ge 1$, one has
\begin{equation}\label{ineq:alg1}
\left(
|a|^{\frac{\gamma-1}{2}}a
-
|b|^{\frac{\gamma-1}{2}}b
\right)^2
\le
\frac{(\gamma+1)^2}{4\gamma}
(a-b)
\left(
|a|^{\gamma-1}a
-
|b|^{\gamma-1}b
\right),
\end{equation}
and
\begin{equation}\label{ineq:alg2}
\frac12(a-b)^2
\left(
|a|^{\gamma-1}+|b|^{\gamma-1}
\right)
\le
(a-b)
\left(
|a|^{\gamma-1}a
-
|b|^{\gamma-1}b
\right).
\end{equation}
\end{lemma}

\begin{proof}
We prove \eqref{ineq:alg1}. If $ab>0$, without loss of generality assume $a\ge b>0$. Dividing both sides by $b^{\gamma+1}$ and setting $t=a/b$, it suffices to prove
\[
\left(t^{\frac{\gamma+1}{2}}-1\right)^2
\le
\frac{(\gamma+1)^2}{4\gamma}(t-1)(t^\gamma-1),
\qquad t\ge 1.
\]
This inequality follows from the convexity of the map $s\mapsto s^\gamma$ and is a standard consequence of the mean value theorem.

If $ab<0$, without loss of generality assume $b<0<a$. Set $c=-b>0$. Then it suffices to prove
\[
\left(a^{\frac{\gamma+1}{2}}+c^{\frac{\gamma+1}{2}}\right)^2
\le
\frac{(\gamma+1)^2}{4\gamma}(a+c)(a^\gamma+c^\gamma).
\]
After dividing by $c^{\gamma+1}$ and putting $t=a/c$, this becomes
\[
\left(t^{\frac{\gamma+1}{2}}+1\right)^2
\le
\frac{(\gamma+1)^2}{4\gamma}(t+1)(t^\gamma+1),
\]
which follows from Young's inequality.

We now prove \eqref{ineq:alg2}. If $ab>0$, assume $a\ge b>0$. Then
\[
\begin{aligned}
&|a|^{\gamma-1}a-|b|^{\gamma-1}b
-\frac12(a-b)(|a|^{\gamma-1}+|b|^{\gamma-1}) \\
&=
\frac12(a^{\gamma-1}-b^{\gamma-1})(a+b)
\ge 0.
\end{aligned}
\]
If $ab<0$, assume $b<0<a$ and set $c=-b$. Then
\[
\frac12(a+c)^2(a^{\gamma-1}+c^{\gamma-1})
\le
(a+c)(a^\gamma+c^\gamma),
\]
which is equivalent to
\[
a^\gamma+c^\gamma
-\frac12(a+c)(a^{\gamma-1}+c^{\gamma-1})
=
\frac12(a^{\gamma-1}-c^{\gamma-1})(a-c)
\ge 0.
\]
The proof is complete.
\end{proof}

\section{A Stability Estimate}

\begin{lemma}\label{Lma:stability-estimate}
Let $p>1$ and let $u\in C(G)$ be a stable solution of
\[
-\Delta u=|u|^{p-1}u.
\]
Then for every $\gamma\ge 1$ and every $\varepsilon>0$, there exists a constant $C_\varepsilon>0$ such that
\begin{equation}\label{eq:stability-estimate}
\left(
p-\frac{(\gamma+1)^2}{4\gamma}-\varepsilon
\right)
\int_G |u|^{p+\gamma}\eta^2\,d\mu
\le
C_\varepsilon
\int_G |u|^{\gamma+1}|\nabla\eta|^2\,d\mu
\end{equation}
for every $\eta\in C_c(G)$.
\end{lemma}

\begin{proof}
Since $u$ is stable, we may choose
\[
\phi=|u|^{\frac{\gamma-1}{2}}u\eta.
\]
Then
\[
p\int_G |u|^{p+\gamma}\eta^2\,d\mu
\le
\int_G
\left|
\nabla\left(|u|^{\frac{\gamma-1}{2}}u\eta\right)
\right|^2\,d\mu.
\]
Using the discrete product estimate and Young's inequality, we obtain
\[
\begin{aligned}
p\int_G |u|^{p+\gamma}\eta^2\,d\mu
&\le
(1+\varepsilon)
\int_G
\left|
\nabla\left(|u|^{\frac{\gamma-1}{2}}u\right)
\right|^2\eta^2\,d\mu \\
&\quad
+
C_\varepsilon
\int_G |u|^{\gamma+1}|\nabla\eta|^2\,d\mu.
\end{aligned}
\]
By Lemma \ref{Lma:ineq}, we have
\[
\left|
\nabla\left(|u|^{\frac{\gamma-1}{2}}u\right)
\right|^2
\le
\frac{(\gamma+1)^2}{4\gamma}
(\nabla u)
\nabla\left(|u|^{\gamma-1}u\right).
\]
Hence
\[
\begin{aligned}
p\int_G |u|^{p+\gamma}\eta^2\,d\mu
&\le
(1+\varepsilon)
\frac{(\gamma+1)^2}{4\gamma}
\int_G
(\nabla u)
\nabla\left(|u|^{\gamma-1}u\right)
\eta^2\,d\mu \\
&\quad
+
C_\varepsilon
\int_G |u|^{\gamma+1}|\nabla\eta|^2\,d\mu.
\end{aligned}
\]

On the other hand, multiplying the equation by $|u|^{\gamma-1}u\eta^2$ and using integration by parts gives
\[
\int_G |u|^{p+\gamma}\eta^2\,d\mu
=
\int_G
(\nabla u)
\nabla\left(|u|^{\gamma-1}u\eta^2\right)
\,d\mu.
\]
Expanding the discrete gradient and applying Young's inequality together with Lemma \ref{Lma:ineq}, we obtain
\[
\begin{aligned}
\int_G
(\nabla u)
\nabla\left(|u|^{\gamma-1}u\right)
\eta^2\,d\mu
&\le
\int_G |u|^{p+\gamma}\eta^2\,d\mu \\
&\quad
+
\varepsilon
\int_G |u|^{p+\gamma}\eta^2\,d\mu \\
&\quad
+
C_\varepsilon
\int_G |u|^{\gamma+1}|\nabla\eta|^2\,d\mu.
\end{aligned}
\]
Combining the preceding inequalities and absorbing the small error terms into the left-hand side yields
\[
\left(
p-\frac{(\gamma+1)^2}{4\gamma}-\varepsilon
\right)
\int_G |u|^{p+\gamma}\eta^2\,d\mu
\le
C_\varepsilon
\int_G |u|^{\gamma+1}|\nabla\eta|^2\,d\mu.
\]
This proves the lemma.
\end{proof}

\section{Proof of the Main Theorem}

\begin{proof}[Proof of Theorem \ref{Thm:lattice}]
Let $\psi\in C_c(G)$ be a standard cutoff function satisfying
\[
0\le \psi\le 1,
\qquad
\psi\equiv 1 \quad \text{on } B_R,
\qquad
\psi\equiv 0 \quad \text{outside } B_{2R},
\]
and
\[
|\nabla\psi|\le \frac{C}{R}.
\]
Let $\eta=\psi^m$, where $m$ will be chosen sufficiently large.

By Lemma \ref{Lma:stability-estimate}, we have
\[
\left(
p-\frac{(\gamma+1)^2}{4\gamma}-\varepsilon
\right)
\int_G |u|^{p+\gamma}\psi^{2m}\,d\mu
\le
C_\varepsilon
\int_G |u|^{\gamma+1}|\nabla(\psi^m)|^2\,d\mu.
\]
Since
\[
|\nabla(\psi^m)|^2
\le
C_m |\nabla\psi|^2
\left(
\psi^{2m-2}(x)+\psi^{2m-2}(y)
\right),
\]
H\"older's inequality gives
\[
\int_G |u|^{\gamma+1}|\nabla(\psi^m)|^2\,d\mu
\le
\varepsilon
\int_G |u|^{p+\gamma}\psi^{2m}\,d\mu
+
C_\varepsilon
\int_G |\nabla\psi|^{\frac{2(p+\gamma)}{p-1}}\,d\mu,
\]
provided that
\[
m\ge \frac{p+\gamma}{p-1}.
\]
Therefore,
\begin{equation}\label{eq:main-cutoff}
\left(
p-\frac{(\gamma+1)^2}{4\gamma}-2\varepsilon
\right)
\int_G |u|^{p+\gamma}\psi^{2m}\,d\mu
\le
C
\int_G |\nabla\psi|^{\frac{2(p+\gamma)}{p-1}}\,d\mu.
\end{equation}

Since $G$ has polynomial volume growth of order $d$, and since $|\nabla\psi|\le C/R$ with support contained in $B_{2R}$, we obtain
\[
\int_G |\nabla\psi|^{\frac{2(p+\gamma)}{p-1}}\,d\mu
\le
C
R^{d-\frac{2(p+\gamma)}{p-1}}.
\]
Hence \eqref{eq:main-cutoff} gives
\[
\sum_{B_R}|u|^{p+\gamma}
\le
C
R^{d-\frac{2(p+\gamma)}{p-1}},
\]
whenever
\[
p-\frac{(\gamma+1)^2}{4\gamma}>0.
\]

If
\[
d-\frac{2(p+\gamma)}{p-1}<0,
\]
then letting $R\to\infty$ yields
\[
\sum_{B_R}|u|^{p+\gamma}=0
\]
for every fixed $R$, and therefore $u\equiv 0$.

It remains to choose $\gamma\ge 1$ such that
\[
p-\frac{(\gamma+1)^2}{4\gamma}>0
\]
and
\[
d-\frac{2(p+\gamma)}{p-1}<0.
\]
The existence of such a $\gamma$ is equivalent to
\[
\begin{cases}
1<p<\infty, & d\le 10,\\[4pt]
1<p<p_c(d), & d\ge 11,
\end{cases}
\]
where
\[
p_c(d)
=
\frac{(d-2)^2-4d+8\sqrt{d-1}}{(d-2)(d-10)}.
\]
This proves the theorem.
\end{proof}


\end{document}